\documentclass{amsart}

\usepackage{eufrak}
\usepackage{amssymb}
\usepackage{mathrsfs}
\usepackage{accents}
\usepackage{hyperref}
\usepackage{amsmath}

\newtheorem{theorem}{Theorem}[section]
\newtheorem{lemma}[theorem]{Lemma}

\theoremstyle{definition}
\newtheorem{definition}[theorem]{Definition}

\theoremstyle{remark}

\numberwithin{equation}{section}

\begin{document}

\title{Hochschild Cohomological Dimension is Not Upper Semi-Continuous}

\author{Anastasis Kratsios}
\address{Department of Mathematics, ETH Z\"{u}rich, HG G 32.3, R\"{a}mistrasse 101, 8092 Z\"{u}rich.}
\email{anastasis.kratsios@math.ethz.ch}


\subjclass[2000]{Primary 16E40; Secondary 16E40, 13C15}

\date{\today}


\keywords{Dimension Theory, Hochschild Cohomological Dimension, Algebraic Geometry, Non-Commutative Algebraic Geometry.}

\begin{abstract}
	It is shown that the Hochschild Cohomological dimension of an associative algebra is not an upper-semi continuous function, showing the semi-continuity theorem is no longer valid for non-commutative algebras.  
	A family of $\mathbb{C}$ exhibits this-algebras parameterized by $\mathbb{C}$ all but one of which has Hochschild cohomological dimension $2$ and the other having Hochschild cohomological dimension $1$.  
\end{abstract}

\maketitle


\section*{Introduction}
Typically, definitions of dimension are upper semi-continuous.  This, for example, is true of the topological dimension of a manifold and the Krull dimension of a commutative ring; or its corresponding scheme.  In this paper, it is shown that this fails to hold in the category of non-commutative $k$-algebras, over a commutative ring $k$.  

Specifically, we demonstrate, by means of a counter-example, that this is not the case for the Hochschild cohomological dimension over such a $k$-algebra.  In this paper, all rings and algebras are assumed to be unital and associative.  Denote by $HH^n(A,M)$ the $n^{th}$ Hochschild cohomology of $A$ with coefficients in $M$.  
\setcounter{section}{1}

\begin{definition}
	
	\textit{The Hochschild cohomological dimension} of a $\mathbb{C}$-algebra $A$, denoted $HCDim(A)$
	is defined as the largest natural number $n$ for which there exists an $(A,A)$-bimodule $M$ such that $HH^n(A,M)\not\cong 0$.  If no such number exists then $HCDim(A):=\infty$ \cite{cuntz1995algebra}.   
\end{definition}

The Hochschild cohomological dimension of the family of $\mathbb{C}$-algebras: 
$$
A_a:=\mathbb{C}<x,y>/(xay-ayx-x)
$$
parameterized by $a \in \mathbb{C}$ is the focus of this paper.  

\begin{theorem}\label{theorem1}
	There exists a family $\mathbb{C}$-algebras $A_a$, parameterized by $\mathbb{C}$ such that if $a\neq 0$ then $HCDim(A_a)=2$ and $HCDim(A_0)=1$.  
\end{theorem}
The proof of Theorem~\ref{theorem1} relies on the following Lemma.  
\begin{lemma} $\label{lemmmmm1}$
	$HCDim(A_1)=2$.  
\end{lemma}
\begin{proof} The Poincar\'{e}-Birkhoff-Witt theorem implies $A_1$ is isomorphic to the universal enveloping algebra of the $2$-dimension complex lie algebra $\mathfrak{g}$ with $\mathbb{C}$-basis $\{ x,y \}$ who's Lie bracket relations is described on basis elements as $[x,x]=[y,y]=0$ and $[x,y]=-[y,x]=x$ \cite{2001first}.  
	Therefore there is an $A_1$-projective resolution of $A_1$ of length $2$ \cite{weibel2008introduction}:
	\begin{equation}
	0 \leftarrow A_1 \overset{d^0}{\leftarrow} A_1 \otimes_{\mathbb{C}} \mathbb{C}^2 \overset{d^1}{\leftarrow} A_1 \otimes_{\mathbb{C}} \mathbb{C} \leftarrow 0 
	\label{resollluttionnnn1}
	\end{equation}
	Hence, there are isomorphisms:
	\begin{equation}
	HH^1(A_1,A_1)\cong Ext^1_{A_1^e}(A_1,A_1)\cong
	Hom_{\mathbb{C}}(A_1,A_1);
	\label{eqqqq2111}
	\end{equation}
	where the first isomorphisms are described in \cite{weibel2008introduction} and the last arise from the calculation of the homology of $Hom_{\mathbb{C}}(-,Hom_{\mathbb{C}}(A_1,A_1))$ via to the $A_1$-projective resolution~\eqref{resollluttionnnn1}.  
	
	Since $A_1$ is an infinite dimension $\mathbb{C}$-vector space, $Hom_{\mathbb{C}}(A_1,A_1)\not\cong 0$, whence $HH^1(A_1,A_1)\not\cong 0$.  Therefore, $1\leq HCDim(A)$.  Next, the resolution $\eqref{resollluttionnnn1}$ implies that for $n>1$ there are isomorphisms: 
	\begin{equation*}
	HH^n(A_1,N)\cong Ext^n_{A_1^e}(A_1,N)\cong Ext^n_{A_1}(A_1,Hom_{\mathbb{C}}(A_1,N))
	\cong 0
	\end{equation*}
	From which it follows that, $1\leq HCDim(A)<2$.  
\end{proof}
\begin{proof}[Proof of Theorem~\ref{theorem1}]
	$A_0 = \mathbb{C}<x,y>/(x)\cong \mathbb{C}<y>\cong\mathbb{C}[x]$.  Since $\mathbb{C}[x]$ is a regular $\mathbb{C}$-algebra of finite type it satisfies Van Den Bergh duality in dimension $1$ \cite{van1998relation}; together with the Hochschild-Kronstadt-Rosenberg theorem \cite{weibel2008introduction} this means there are natural isomorphisms for every $(\mathbb{C}[x],\mathbb{C}[x])$-bimodule $M$: \begin{equation*}
		HH^{n}(\mathbb{C}[X], M) \cong HH_{1-n}(\mathbb{C}[X], Der_{\mathbb{C}}(\mathbb{C}[x],\mathbb{C}[x])\otimes_{\mathbb{C}[x]} M) 
		\cong HH_{1-n}(\mathbb{C}[X], M)
		.
		\end{equation*} 
		In particular, this implies the Hochschild cohomology vanishes for $n>1$ and 
		\begin{equation*}
		HH^1(\mathbb{C}[X],\mathbb{C}[X])\cong HH_0(\mathbb{C}[X],\mathbb{C}[x])\cong Z(\mathbb{C}[X])=\mathbb{C}[X]. 
		\end{equation*}
		Thus, $HCDim(A_0)=1$.  

On the other hand if $a\neq 0$, $A_a$ is isomorphic to $A_1:=\mathbb{C}\langle x,y\rangle/(xy-yx-x)$ via the $\mathbb{C}$-algebra isomorphisms $\psi_a: A_a \rightarrow A_1$ mapping $y \mapsto ay$ (with inverse $y \mapsto \frac{1}{a}y$).  

Hochschild cohomology is functorial \cite{weibel2008introduction}, therefore for every $(A,A)$-bimodule $M$, for every $n\in \mathbb{N}$ and $a\in \mathbb{C}- \{ 0 \}$ $\psi_a$ induces isomorphisms 
\begin{equation*}
HH^n(\psi_a,M): HH^n(A_1,M) \rightarrow HH^n(A_a,M).  
\label{hcdimmmm1isos}
\end{equation*}  
Whence $HCDim(A_a)=HCDim(A_1)$ for all $a\in \mathbb{C}- \{ 0 \}$ .  
Therefore it is enough to compute the Hochschild cohomological dimension of $A_1$.  Hence, Lemma~\ref{lemmmmm1} entails $HCDim(A_a)=2$ for all $a\in \mathbb{C}-\{ 0 \} $.  Hence $HCDim(A_a)=2$ if $a \in \mathbb{C}- \{ 0 \}$.  
\end{proof}
\section{Conclusion}
In this short article, we have shown that the Hochschild cohomological dimension does not share some of the critical properties of most standard algebraic and topological dimensions.  Namely, we have demonstrated the existence of a family of algebras which fail to exhibit the familiar \textit{upper semi-continuity of dimension} property.  The construction of this counter-example hinged on within the category of non-commutative algebras.  
\nocite{*}
------------------------
\bibliographystyle{amsplain}

\end{document}